\documentclass{article}
\usepackage{latexsym}
\usepackage{multirow}
\usepackage{amsmath}
\usepackage{amssymb}
\usepackage{amsfonts}
\usepackage{bbm}
\usepackage[T1]{fontenc}
\usepackage{amsthm}
\usepackage{enumerate}
\usepackage{babel}
\usepackage{tikz}
\usetikzlibrary{matrix,arrows,decorations.pathmorphing}
\usepackage{tikz-cd}
\usepackage[utf8]{inputenc}
\usepackage{cite}
\usepackage{mathrsfs}
\usepackage{mathtools}
\usepackage{calligra}

\usepackage{hyperref}
\usepackage{todonotes}

\date{\today}
\title{Initial values of ML-degree polynomials}
\author{Maciej Ga\l{}\k{a}zka}
\date{\today}
\newtheorem{theorem}{Theorem}[section]
\newtheorem{proposition}[theorem]{Proposition}
\newtheorem{lemma}[theorem]{Lemma}
\newtheorem{conjecture}[theorem]{Conjecture}

\newtheorem{corollary}[theorem]{Corollary}

\theoremstyle{definition}

\newtheorem{definition}[theorem]{Definition}
\theoremstyle{remark}

\newcommand{\Pf}{\operatorname{Pf}}

\DeclareMathOperator{\ShHom}{\mathscr{H}\text{\kern -3pt {\calligra\large om}}\,}
\DeclareMathOperator{\ShDer}{\mathscr{D}\text{\kern -3pt {\calligra\large er}}\,}
\DeclareMathOperator{\ShSpec}{\mathscr{S}\text{\kern -3pt {\calligra\large pec}}\,}

\renewcommand{\mathcal}[1]{\mathscr{#1}}

\begin{document}
\author{Maciej Ga\l{}\k{a}zka\thanks{ Institute of Information Technology, Warsaw University of Life Sciences, ul.
  Nowoursynowska 159, 02-776, Warsaw, Poland, \\ \hspace{10cm} Universit\`a di Trento, via Sommarive, 38123 Povo 14, TN, Italy}
}
\maketitle
\abstract{We prove a conjecture about the initial values of ML-degree polynomials stated by
Micha{\l}ek, Monin, and Wi\'sniewski.}
\section{Introduction}
Maximum likelihood degree is one of the central invariants in algebraic statistics
\cite{sturmfels_uhler}, \cite{sullivant}. Recently, for generic linear concentration models Manivel
et al.\ \cite{complete_quadrics} proved polynomiality of the maximum likelihood degree function
$\phi(n,d)$ for fixed $d$. The proof relied on relations of ML-degree to other branches of
mathematics:
\begin{itemize}
  \item intersection theory \cite{complete_quadrics_procesi} (Segre classes  of the second symmetric power of the universal bundle
    over the Grassmannian),
  \item representation theory,
  \item semidefinite programming (algebraic degree of semidefinite programming
    \cite{nie_ranestad_sturmfels}),
  \item theory of symmetric polynomials \cite{pragacz_book},\cite{laksov_lascoux_thorup}.
\end{itemize}
This work may be also regarded as a symmetric matrix parallel of the groundbreaking work
\cite{adiprasito_huh_katz} by Adiprasito, Huh, Katz  and \cite{milnor_numbers_huh} by Huh, which
would correspond to diagonal matrices. From this perspective, $\phi(n, d)$ is the analogue of the
$\beta$-invariant of a uniform matroid, which is simply the binomial coefficient
\cite{applications_of_intersection_theory}.

The function $\phi(n,d)$ has  a very nice enumerative geometric definition as the number of quadrics in $n$ variables
that pass through $d$ general points and are tangent to $\binom{n+1}{2} - d-1$ general hyperplanes. Therefore,
polynomiality results can be quite surprising, as each time $n$ changes, so does the ambient space. However, once it is
established, it allows us to extend the scope of parameters where the question is meaningful to arbitrary integral $n$.
If $n$ is positive, but $d > \binom{n+1}{2} -1$, nothing interesting happens, both the geometric interpretation, and the
polynomial have value $0$.  Following a well-established program in mathematics called combinatorial reciprocity
\cite{combinatorial_reciprocity_theorems}, one can proceed to see what happens for $n \leq 0$.

In \cite[page 83]{maximum_likelihood_degree} simple formulas were conjectured for
$\phi(0,d)$ and $\phi(-1,d)$.
\begin{theorem}
  \label{theorem:main_introduction}
  \begin{align*}
    \phi(0, d) &= (-1)^{d-1}\text{,} \\
    \phi(-1, d) &= (-2)^{d-1}\text{.}
  \end{align*}
\end{theorem}
The goal of the article is to prove Theorem \ref{theorem:main_introduction}. Additionally, in
Section \ref{section:conjecture}, we state a conjecture about the values $\phi(-n, d)$ for $n \geq 1$.

Our tools include Lascoux coefficients $\psi_I$ and Lascoux polynomials $LP_I(n)$. The former appear
for instance in \cite[Proposition A.15]{laksov_lascoux_thorup}, \cite[Example
7.10]{pragacz_enumerative}. They can be defined for any finite set of non-negative integers $I$
using the following recursive formula
\begin{align}\label{equation:psi_definition}
  \begin{split}
  \psi_{\{i\}} &= 2^i\text{,} \\
  (r+1)\psi_{\{j_0,\dots,j_r\}} - 2\sum_{l = 1}^r \psi_{\{j_0,\dots,j_l -1,\dots,j_r\}} &=
  \begin{cases}\psi_{\{j_1,\dots,j_r\}} & \text{ if } j_0 = 0\text{,} \\ 0 & \text{
  otherwise,}\end{cases} 
\end{split}
\end{align}
where $0 \leq j_0 < j_1 <\dots< j_r$ and the sum is over all $l$ for which $j_l -1 > j_{l-1}$.

Let $[n]$ denote the set $\{0,1,\dots,n-1\}$. In \cite[Theorem 4.3]{complete_quadrics}, the authors noticed that the function given by the
formula
\begin{equation*}
  LP_I(n) = \begin{cases}\psi_{[n]\setminus I} &  \text{ if } I\subseteq [n]\text{,} \\
  0 & \text{ otherwise.}\end{cases}
\end{equation*}
is a polynomial for every $I$. Following them, we call it the Lascoux polynomial.

The following formula, which is proven in \cite[Corollary 3.6, Theorem 3.7]{complete_quadrics}, expresses the ML degree
polynomial in terms of the Lascoux coefficients and Lascoux polynomials.
  \begin{equation}
    \label{equation:formula_for_phi}
    \phi(n, d) = \sum_{s: \binom{s+1}{2} \leq d} \frac{s}{n}\sum_{\# I = s, \Sigma I = d-s} \psi_I
    LP_I(n)
  \end{equation}
  Here $I$ is always a set of non-negative integers, and the numbers $\#I$ and $\Sigma I$ are the cardinality and the sum of $I$, respectively.
\subsection{Acknowledgments}
The author would like to thank Mateusz Micha{\l}ek for suggesting this project and for constant
support. The author was supported by the National Science Center, Poland, project number
2017/26/E/ST1/00231, and the University of Konstanz, by which he was employed during the months of
June--July 2022.

Funded by the European Union under NextGenerationEU. PRIN 2022 Prot. n. 2022ZRRL4C\_004. Views and opinions expressed are
however those of the author(s) only and do not necessarily reflect those of the European Union or European Commission.
Neither the European Union nor the granting authority can be held responsible for them.

\section{Calculations}
We begin by stating useful formulas. The following proposition and its two corollaries are
well known. The first one is a version of the result \cite[Proposition 7.12]{pragacz_enumerative}
stated for Lascoux polynomials
\begin{proposition}[Lemma 5.7 of \cite{complete_quadrics}]
  Let $I = \{i_1,\dots,i_r\}$ be an increasing sequence of non-negative integers. Then
  \begin{align*}
  LP_I(n) &= \Pf((LP_{\{i_k, i_l\}}(n))_{0 < k,l \leq r}) \text{ for even } \#I\text{,} \\
  LP_I(n) &= \Pf((LP_{\{i_k, i_l\}}(n))_{0 \leq k,l \leq r}) \text{ for odd } \#I\text{,}
  \end{align*}
  where $LP_{\{i_0,i_k\}}(n)$ is by definition $LP_{\{i_k\}}(n)$, and $\Pf$ is the Pfaffian. 
  \label{proposition:pfaffians}
\end{proposition}
\begin{corollary}
  For an increasing sequence of non-negative integers $\{i_1,\dots,i_r\}$
  \begin{equation*}
    n^{\lceil \frac{r}{2}\rceil} \mid LP_{\{i_1,\dots,i_r\}}(n)\text{.}
  \end{equation*}
  \label{corollary:high_power_n_divides}
\end{corollary}
\begin{proof}
  Straightforward, because by $n \mid LP_{\{i_k\}}(n)$ and $ n \mid LP_{\{i_k, i_l\}}(n)$ by definition.
\end{proof}
\begin{corollary}
  For an increasing sequence of non-negative integers $\{i_1,\dots,i_r\}$
  \begin{equation*}
    (n+1)^{\lfloor \frac{r}{2}\rfloor} \mid LP_{\{i_1,\dots,i_r\}}(n)\text{.}
  \end{equation*}
  \label{corollary:high_power_n_plus_one_divides}
\end{corollary}
\begin{proof}
  Here $n+1 \mid LP_{\{i_k, i_l\}}$ by one of the formulas from \cite[Example 4.4]{complete_quadrics}.
\end{proof}
As a warm-up, we calculate the value $\phi(-1,d)$.
\begin{proposition}
  \begin{equation*}
    \phi(-1,d) = (-2)^{d-1}\text{.}
  \end{equation*}
  \label{propostion:value_minus_one}
\end{proposition}
\begin{proof}
  Due to Corollary \ref{corollary:high_power_n_plus_one_divides} Equation \eqref{equation:formula_for_phi} has a simple form:
  \begin{equation*}
    \phi(-1,d) = \frac{1}{n} \psi_{\{d-1\}} LP_{\{d-1\}}(-1)
    = 2^{d-1} \frac{(-2)(-3)\cdots(-(d-1)-1)}{d!} = (-2)^{d-1}\text{.}
  \end{equation*}
\end{proof}

Since the polynomial $LP_I(n)$ is divisible by $n$ for $I \neq \emptyset$, we introduce the
following definition.
\begin{definition}
  If $I\neq \emptyset$, let
  \begin{equation*}
    \widetilde{LP}_I(n) = \frac{{LP}_I(n)}{n}\text{.}
  \end{equation*}
  \label{definition:reduced_lascoux_polynomial}
\end{definition}
We also define for $i, j \in \mathbb{N}$:
\begin{definition}
  \begin{equation*}
    C(i,j) = \frac{(j-i)(-1)^{i+j}}{(i+1)(j+1)\binom{i+j+2}{i+1}}\text{.}
  \end{equation*}
  \label{definition:mc}
\end{definition}
This is motivated by the following proposition.
\begin{proposition}
  For a non-negative integer $i$
  \begin{equation*}
    \widetilde{LP}_{\{i\}}(0) = \frac{(-1)^i}{i+1}\text{.}
  \end{equation*}
  For two non-negative integers $i,j$ with $i < j$
  \begin{equation*}
    \widetilde{LP}_{\{i,j\}}(0) = C(i,j)\text{.}
  \end{equation*}
  \label{proposition:value_at_zero}
\end{proposition}
\begin{proof}
  
By \cite[Example 4.4]{complete_quadrics}
\begin{equation*}
  \widetilde{LP}_{\{i\}}(0) = \frac{(-1)(-2)\cdots(-i)}{(i+1)!} =  \frac{(-1)^i}{i+1}\text{,}
\end{equation*}
which proves the first assertion.

By another formula in the same example
\begin{equation*}
  LP_{\{i,j\}}(n) = \frac{(j-i)[n+1]_{j+2}}{(i+1)!(j+1)!(i+j+2)!} \sum_{d=0}^i (-1)^da_{i,d}(i+j+1-d)![n]_{i-d}\text{,}
\end{equation*}
where $a_{i,d} = \prod_{k=0}^{d-1}(i-k)(i-k+1)$ and $[n]_d = n(n-1)\cdots(n-d+1)$. When we set $n=
0$ in the sum 
\begin{equation}
  \sum_{d=0}^i (-1)^d a_{i,d}(i+j+1-d)![n]_{i-d}\text{,}
  \label{equation:sum}
\end{equation}
it is non-zero only for $i = d$. Therefore, the sum \eqref{equation:sum} becomes equal to
$(-1)^i(j+1)!i!(i+1)!$. Hence,
\begin{align*}
  \widetilde{LP}_{\{i,j\}}(0)&= \frac{(j-i) j!(-1)^j}{(i+1)!(j+1)!(i+j+2)!} (-1)^i i! (i+1)!(j+1)! \\
  &= \frac{(j-i)(-1)^{i+j}}{(i+1)(j+1)\binom{i + j+2}{i+1}}\text{.}
\end{align*}
\end{proof}
We leave the proof of the following easy lemma to the reader.
\begin{lemma}
  \label{lemma:mc_formula}
  We have
  \begin{equation}
    C(i+1,j) + C(i, j+1) = -C(i,j)
  \end{equation}
for every non-negative $i,j$.
\end{lemma}

In order to make the proof as simple as possible, we extend the definition of $\psi$ to possibly non-ascending sequences of
numbers of length $2$.
\begin{proposition}
  \label{proposition:identity}
  If we define 
  \begin{enumerate}[(i)]
    \item $\psi_{a,b} = \psi_{\{a,b\}}$ for $a,b \in \mathbb{N}$ with $a < b$,
    \item $\psi_{a,a} = 0$ for $a \in \mathbb{N}$,
    \item $\psi_{a,b} = -\psi_{b,a}$ for $a, b\in \mathbb{N}$ with $a > b$,
    \item $\psi_{-1, a} = -\psi_{a, -1} = 2^{a-1}$ for $a \in \mathbb{N}$,
  \end{enumerate}
  the following identity holds for all $a, b \in \mathbb{N}$
  \begin{equation}
    \label{equation:crucial_identity}
    \psi_{a,b} = \psi_{a-1,b} + \psi_{a,b-1}\text{.}
  \end{equation}
\end{proposition}
\begin{proof}
  By Points (i), (ii) and Equation \eqref{equation:psi_definition}, the identity holds for $1 \leq a <
  b$. By Points (ii) and (iii), it holds for $1 \leq a = b$. Assume $a > b \geq 1$. Then
  \begin{equation*}
    \psi_{a,b} = -\psi_{b,a} = -\psi_{b-1,a} -\psi_{b, a-1} = \psi_{a, b-1} + \psi_{a-1,b}\text{.}
  \end{equation*}
  This proves the identity for $a > b \geq 1$. Suppose at least one of the integers $a, b$ is equal
  to $0$. But then Equation \eqref{equation:crucial_identity} follows from the fact that $\psi_{0,
  a} = 2^a - 1$ for $a \in \mathbb{N}$ and from Point (iv).
\end{proof}
\begin{proposition}
  \label{proposition:sum}
  The following formula holds
  \begin{equation*}
    \sum_{i = 0}^{d+1}\psi_{i,d+1-i} C(i, d+1-i) + \sum_{i=0}^d \psi_{i,d-i} C(i,d-i) =
    \frac{2^{d+1}(-1)^{d+1} (d+1)}{(d+2)(d+3)}\text{.}
  \end{equation*}
\end{proposition}
\begin{proof}
  We have
  \begin{align*}
    \sum_{i=0}^{d+1} \psi_{i,d+1-i} C(i, d+1 -i) \stackrel{\text{by Proposition
    \ref{proposition:identity}}}{=} \hspace{0.1cm} &\sum_{i=0}^{d+1} (\psi_{i-1,d+1-i} +
    \psi_{i,d-i}) C(i, d+1 -i) \\
    \stackrel{\text{changing the indices}}{=} &\sum_{i=-1}^d \psi_{i,d-i} C(i+1, d-i) + \sum_{i = 0}^{d+1} \psi_{i, d-i} C(i,d+1-i) \\
    \stackrel{\text{regrouping}}{=}\hspace{0.57cm} &\sum_{i=0}^d \psi_{i,d-i}(C(i+1, d-i) + C(i,d+1-i)) \\ &+ \psi_{-1,d+1}C(0,d+1) +
    \psi_{d+1,-1}C(d+1,0) \\
    \stackrel{\text{by Lemma \ref{lemma:mc_formula}}}{=} -&\sum_{i=0}^d\psi_{i,d-i} C(i, d-i) +\frac{2^{d+1}(-1)^{d+1} (d+1)}{(d+2)(d+3)} \text{.}
  \end{align*}
\end{proof}
Finally, we are ready to prove the main result of this paper.
\begin{theorem}
  \label{theorem:zero}
  \begin{equation*}
    \phi(0,d) = (-1)^{d-1}\text{.}
  \end{equation*}
\end{theorem}
\begin{proof}
  By Equation \eqref{equation:formula_for_phi} and Corollary \ref{corollary:high_power_n_divides} we get
  \begin{align*}
    \phi(0,d) &= \psi_{\{d-1\}}\widetilde{LP}_{\{d-1\}}(0) + 2\sum_{0\leq i < \frac{d-2}{2}}
    \psi_{\{i,d-2-i\}} \widetilde{LP}_{\{i,d-2-i\}}(0) \\ &= \frac{(-2)^{d-1}}{d} + 2\sum_{0\leq i <
    \frac{d-2}{2}} \psi_{\{i,d-2-i\}} C(i, d-2-i)\text{.}
  \end{align*}
  Hence, it suffices to prove that
  \begin{equation*}
    2\sum_{0\leq i < \frac{d-2}{2}} \psi_{\{i,d-2-i\}} C(i, d-2-i) = (-1)^{d-2}\frac{2^{d-1} -
    d}{d}\text{,}
  \end{equation*}
  or equivalently, using the notation from the statement of Proposition
  \ref{proposition:identity}, that
  \begin{equation}
    \label{equation:final_induction}
    \sum_{i=0}^{d} \psi_{i,d-i} C(i, d-i) = (-1)^d \frac{2^{d+1} - (d+2)}{d+2}\text{.}
  \end{equation}
  We do this by induction. Assume we know \eqref{equation:final_induction} holds for a given $d$.
  Then by Proposition \ref{proposition:sum}
  \begin{align*}
    \sum_{i=0}^{d+1} \psi_{i,d+1-i} C(i, d+1-i) &= - \sum_{i=0}^d \psi_{i,d-i}C(i,d-i) +
    2^{d+1}\frac{(-1)^{d+1}(d+1)}{(d+2)(d+3)} \\
    &= (-1)^{d+1}\frac{2^{d+1} - (d+2)}{d+2} + 2^{d+1}\frac{(-1)^{d+1}(d+1)}{(d+2)(d+3)}\\
    &= (-1)^{d+1} \frac{2^{d+2} - (d+3)}{d+3}\text{,}
  \end{align*}
  as desired.
\end{proof}
\section{Conjecture about negative values}
\label{section:conjecture}
The question remains, what about smaller negative values of $n$? The following conjecture is motivated by analyzing many
examples.
\begin{conjecture}
  Let $n$ be a positive integer.
  \begin{equation*}
    \phi(-n,d) = (-1)^{d-1} 2^d\left( \sum_{i = 1}^{\frac{n(n+1)}{2}} \frac{\phi(n,i)}{2^i} \binom{d-1}{i-1}\right)\text{.}
  \end{equation*}
\end{conjecture}
Notice that the conjecture says that up to multiplication by the factor $(-1)^{d-1} 2^d$, the function $d \mapsto \phi(-n,d)$ is a
polynomial in $d$ whose coefficients depend on the values $\phi(n,i)$.

The number $\phi(-n,d)$ is always an integer. Following the combinatorial reciprocity program, one can wonder whether
$|\phi(-n,d)|$ has some interpretation in enumerative geometry, or possibly even in statistics.
\def\polhk#1{\setbox0=\hbox{#1}%
  {\ooalign{\hidewidth\lower1.5ex\hbox{`}\hidewidth\crcr\unhbox0}}}\def\dbar{\leavevmode\hbox
  to 0pt{\hskip.2ex\accent"16\hss}d}\def\cdprime{$''$} \def\cprime{$'$}
  \def\cprime{$'$} \def\cprime{$'$} \def\Dbar{\leavevmode\lower.6ex\hbox to
  0pt{\hskip-.23ex \accent"16\hss}D} \def\cprime{$'$} \def\cprime{$'$}
  \def\cdprime{$''$} \def\Dbar{\leavevmode\lower.6ex\hbox to 0pt{\hskip-.23ex
  \accent"16\hss}D} \def\cprime{$'$} \def\cprime{$'$} \def\cprime{$'$}
  \def\cprime{$'$} \def\Dbar{\leavevmode\lower.6ex\hbox to 0pt{\hskip-.23ex
  \accent"16\hss}D} \def\cprime{$'$} \def\cprime{$'$}

\end{document}